\documentclass[12pt,a4paper]{amsart}

\usepackage{amsmath,amsfonts,amssymb}
\usepackage[mathscr]{eucal}

\usepackage[hmargin=2cm,vmargin=2cm]{geometry}

\usepackage[hidelinks]{hyperref}
\usepackage{nameref,zref-xr}                    

\usepackage{amsthm,amscd}
\usepackage{epsfig}
\usepackage{color}
\usepackage[all]{xy}
\usepackage{tikz-cd}
\usepackage{graphics, setspace}
\usepackage{breqn}
\usepackage{amstext}
\usepackage{array}

\allowdisplaybreaks

\newcommand{\CC}{{\mathbb{C}}}

\newcommand{\QQ}{{\mathbb{Q}}}

\newcommand{\RR}{{\mathbb{R}}}

\newcommand{\calS}{{\cal S}}

\newcommand{\E}[2]{{E_{#1}^{#2}}}
\newcommand{\F}[2]{{F_{#1}^{#2}}}

\newcommand{\Jac}{\mathrm{Jac}}

\newcommand{\bx}{{\bf x}}

\newcommand{\bz}{{\bf z}}
\newcommand{\hess}{\mathrm{hess}}

\newcommand{\res}{\mathrm{res}}
\newcommand{\wt}{\mathrm{wt}}

\def\A{{\mathcal A}}

\def\C{{\mathcal C}}

\def\E{{\mathcal E}}
\def\F{{\mathcal F}}
\def\G{{\mathcal G}}
\def\H{{\mathcal H}}

\def\L{{\mathcal L}}

\def\O{{\mathcal O}}

\def\T{{\mathcal T}}

\renewcommand{\S}{{\mathcal S}}
\def\calS{{\mathcal S}}

\newcommand{\p}{{\partial}}

\newcommand{\eps}{\varepsilon}

\newcommand{\Id}{\mathrm{Id}}

\newcommand{\ZZ}{\mathbb{Z}}

\newcommand{\bs}{{\bf s}}
\newcommand{\PV}{{\mathrm{PV}}}

\newcommand{\GM}{{\mathrm{GM}}}

\newcommand{\End}{{\mathrm{End}}}
\def\ns{{\nabla}\hspace{-1.2mm}{\text{\footnotesize{\bf /}}}}

\newtheorem{theorem}{Theorem}[section]
\newtheorem{proposition}[theorem]{Proposition}
\newtheorem{lemma}[theorem]{Lemma}
\newtheorem{corollary}[theorem]{Corollary}

\newtheorem{example}[theorem]{Example}
\newtheorem{remark}[theorem]{Remark}

\usepackage{color}

\def\&{\vspace{-5pt}&}

\numberwithin{equation}{section}

\begin{document}

\title{Elements of Saito theory via Batalin--Vilkovisky algebras}

\author{Alexey Basalaev}
\address{A. Basalaev:
\newline
Faculty of Mathematics, HSE University, Usacheva str., 6, 119048 Moscow, Russian Federation
}
\email{abasalaev@hse.ru}

\date{\today}

 \begin{abstract}
    Saito theory associates to a quasihomogeneous isolated singularity the structure of a Dubrovin--Frobenius manifold. This structure is not unique, depending on the special choice of a primitive form or, equivalently, a good basis.
    We study primitive forms and respective Dubrovin--Frobenius manifolds via BV-algebras. In particular, we give recursive formulae for the primitive form of K. Saito and the R-matrix of Givental using BV-algebra computations.
 \end{abstract}
 \maketitle


\section{Introduction}

Dubrovin--Frobenius manifolds were introduced by B. Dubrovin and have appeared as an important tool in such areas of mathematics as integrable systems, mathematical physics, and mirror symmetry. The sources of Dubrovin--Frobenius manifolds include formal deformation theory, cohomological field theory, and Saito theory of isolated singularities (cf. \cite{M}). In this text we focus on the last one, building up the connection to the previous two.

\subsection{Dubrovin--Frobenius manifold of a singularity}
A Dubrovin--Frobenius manifold is the set of data $(M,\circ,\eta,e)$, where $M$ is a complex manifold, $\circ$ is a fiberwise $\O_M$--bilinear commutative and associative product on its holomorphic tangent sheaf $\T_M$, $\eta$ is a non--degenerate $\O_M$--bilinear form on $\T_M$, such that the Frobenius algebra property $\eta(X \circ Y, Z) = \eta(X,Y\circ Z)$ holds for any $X,Y,Z \in \T_M$. It is required that the Levi-Civita connection $\ns$ of $\eta$ is flat, and the unit vector field $e$ of the product is $\ns$--flat. Let $t^\alpha$ be flat coordinates and $\p / \p t^\alpha$ the dual vector fields. Denote by $C := \sum_\alpha \left( \frac{\p}{\p t^\alpha} \circ \right) dt^\alpha$ the $\T_M$--endomorphism-valued $1$--form. It is required that $\ns C = 0$. Moreover both $\eta$ and $\circ$ are required to be quasihomogeneous with respect to some Euler vector field $E$.

To any Dubrovin--Frobenius manifold one can associate the connection on $M \times \CC^\ast_z$ given by
\begin{align}\label{eq: Dubrovin connection-1}
    \widetilde \nabla_X := \nabla_X + \frac{1}{z} C_X, \quad \widetilde \nabla_{\frac{d}{dz}} := \frac{d}{dz} + \frac{1}{z} \left( B_0 + \frac{B_\infty}{z} \right)
\end{align}
for $C_X(Y) := X \circ Y$, $B_0(Y) := E \circ Y$ and the diagonal grading operator $B_\infty$.

This new connection is flat. It provides an important piece of data of a Dubrovin--Frobenius manifold. Conversely, given a manifold $M$ with a flat pairing $\eta$ and a fiberwise product $\circ$ on $\T_M$ satisfying the Frobenius algebra property, if $\widetilde \nabla$ is flat, then $M$ is a Dubrovin--Frobenius manifold (cf. Proposition 2.1 in \cite{D2} and \cite{Sa}).

The structure of every Dubrovin--Frobenius manifold can be encoded by just one function $\F = \F(t_1,\dots,t_\mu)$ called its potential.
The associativity of the product $\circ$ then implies that $\F$ is subject to a big system of PDEs called the \textit{WDVV equation}.

\subsubsection{Saito--Frobenius manifold}
Consider a quasihomogeneous polynomial $f \in \CC[x_1,\dots,x_N]$ with $0 \in \CC^N$ --- the only critical point. It defines a Jacobian algebra
\[
 \Jac(f) := \CC[x_1,\dots,x_N] / (\frac{\p f}{\p x_1},\dots \frac{\p f}{\p x_N})
\]
that is a finite--dimensional algebra with a pairing called the \textit{residue pairing}. Let $\phi_1(\bx),\dots, \phi_\mu(\bx) \in \CC[x_1,\dots,x_N]$ be polynomials whose classes generate $\Jac(f)$. Introduce the new function
\[
 F(\bx,\bs) := f(\bx) + \sum_{\alpha=1}^\mu s_\alpha \phi_\alpha(\bx)
\]
with parameters $s_1,\dots,s_\mu$ varying in a small open neighbourhood of the origin $\calS \subset \CC^\mu$. We call $F$ the unfolding of $f$ and $\calS$ the unfolding space.


It is well--known that the unfolding space $\calS$ carries the structure of a Dubrovin--Frobenius manifold (cf. \cite{ST,D1,D2,Sa, H}). To construct it, one considers the $\CC[z]$--module
\[
\H_f^{(0)} := \Omega^N_{\CC^N}[z] / (zd + df \wedge) \Omega^{N-1}_{\CC^N}[z],
\]
called the \textit{Brieskorn lattice}, and its completion $\H_f := \H_f^{(0)} \otimes_{\CC[z]} \CC((z^{-1}))$. Assuming $f$ varies in a family given by the unfolding, we get the $\O_\calS \otimes \CC[z]$--module $\H_F^{(0)}$ and its completion $\H_F$.
Both $\H_f^{(0)}$ and $\H_F^{(0)}$ are free rank $\mu$ modules over $\CC[z]$ and $\O_\calS \otimes \CC[z]$ respectively. However, they also have a $\CC[\bx]$--module structure given by $p(\bx) \cdot [\phi] d^N\bx = [p \phi] d^N\bx$.

The space $\H_F$ is endowed with the connection $\nabla^\GM: \T_\calS \otimes \H_F \to \H_F $ proposed by K. Saito:
\begin{align*}
    & \nabla^\GM_{X} [\phi]d^N\bx := [X \cdot \phi(\bx) + \frac{1}{z} (X \cdot F)\phi(\bx)]d^N\bx,
    \\
    & \nabla^\GM_{\frac{d}{dz}} [\phi]d^N\bx := [\frac{d\phi}{dz} + \frac{1}{z^2} F \phi - \frac{1}{z} N \phi ]d^N\bx,
\end{align*}
where $X \in \T_\calS$ and $X\cdot$ stands for the directional derivative.
This connection looks similar to that of Eq.~\eqref{eq: Dubrovin connection-1}; however, there are two major differences: 1. it is not a connection on $\T_\calS\otimes \CC[z,z^{-1}]]$; 2. its $z$--dependence is much more complicated. These two issues are resolved with the help of K. Saito's \textit{primitive form}.

\subsubsection{Primitive form and good basis}
A primitive form is a special element $\zeta \in \H_F^{(0)}$ such that the map $\sigma: \T_\calS \to \H_F^{(0)} / z \H_F^{(0)}$, $X \mapsto z \nabla^\GM_X$ is an isomorphism and the connection $\sigma^{-1} \cdot \nabla^\GM \cdot \sigma$ has the form of Eq.~\eqref{eq: Dubrovin connection-1} (cf. \cite{SK2,ST}), giving a Dubrovin--Frobenius manifold structure on~$\calS$.

Another way to construct the Dubrovin--Frobenius manifold structure on $\S$ is by taking a special choice of an $\O_\S[z]$--basis $\omega_1,\dots,\omega_\mu$ of $\H_F^{(0)}$, called a \textit{good basis}. Computed in this basis, the connection $\nabla^\GM$ takes the form of Eq.~\eqref{eq: Dubrovin connection-1}.

Existence of a primitive form was proved by M. Saito (cf. \cite{SM1,SM2}). He also showed that primitive forms are in one-to-one correspondence with good bases.

We will denote the Dubrovin--Frobenius manifold of $f$ fixed by the good basis $\omega_\bullet$ by $M_f^\omega$ and its potential by $\F^\omega$.


\subsection{BV--algebra for a hypersurface singularity}
Let $\PV^{i} := \wedge^i \T_{\CC^N}$ and $\PV := \sum_{i \ge 0} \PV^{i}$. It has an algebra structure given by the wedge product. Construction with a holomorphic volume form on $\CC^N$ gives an isomorphism $\PV^i \cong \Omega_{\CC^N}^{N-i,0}$ to the space of holomorphic $(N-i)$--forms. Denote by $\p: \PV^{i} \to \PV^{i-1}$ the operator lifted from the holomorphic differential on $\Omega_{\CC^N}^{\ast,0}$.
Then $\p^2 = 0$, but $\p$ does not satisfy the Leibniz rule, being a 2nd order operator. It defines a Gerstenhaber bracket on $\PV$ by
\[
\lbrace \alpha,\beta \rbrace := \p(\alpha \cdot \beta) - (\p \alpha) \cdot \beta - (-1)^{ |\alpha|} \alpha \cdot \p (\beta).
\]

For $f$ as in the preceding section, consider another differential $\overline\p_f$ on the same space $\PV$. Set $\overline\p_f := \lbrace f, \cdot \rbrace$. Then
\[
    H^\ast(\PV,\overline\p_f) \cong \Jac(f) \quad \text{ and } \quad H^\ast(\PV[z], \overline\p_f + z \p) \cong \H_f^{(0)}.
\]
The operators $\overline\p_f$ and $\p$ serve as a differential and a BV--operator on $\PV$ respectively, giving the BV--algebra $(\PV,\overline{\p}_f,\p)$. A similar BV--algebra was used in  \cite{LLS}, however with antiholomorphic directions as well. We do not use them here, but keep the notation.

With the help of the unfolding function $F$, the BV--algebra above can be extended to a new BV--algebra $(\PV\otimes\O_\S,\overline \p_F,\p)$.
This BV--algebra can be endowed with a connection $\nabla^\PV$ similar to that of $\nabla^\GM$.

Consider the $\PV \otimes \O_\S((z^{-1}))$--operators
\begin{align*}
    & \nabla_v^\PV (a) := v(a) + z^{-1} a v(F), \quad \forall v \in \T_\S,
    \\
    & \nabla_{\frac{\p}{\p z}}^\PV (a) := \frac{\p a}{\p z} - z^{-2} a F.
\end{align*}
Direct computations show that $[\nabla_v^\PV, \overline\p_F + z  \p] = 0$ and $ [\nabla_{\frac{\p}{\p z}}^\PV, \overline \p_F + z \p] = z^{-1} (\overline \p_F + z \p)$, showing that $\nabla^\PV$ is well--defined on the cohomology $H^\ast( \PV \otimes \O_\S((z^{-1})), \bar \p_F + z\p)$.

\subsubsection{BV operator trivialization}
An operator $\Phi \in \End(\PV[z])$ is called a \textit{BV--operator trivialization} of $(\PV,\overline{\p}_f,\p)$ if
\[
    \Phi (\overline{\p}_f + z \p) = \overline{\p}_f \Phi.
\]
In \cite{KMS} Khoroshkin, Markarian and Shadrin constructed a quasi-isomorphism between the hypercommutative operad and the operad of BV--algebras with resolved BV--operator.

Having fixed the trivialization above, their construction allows us to associate to any $k$--tuple of cohomology classes $[\phi_\bullet] \in H^\ast(\PV,\overline \p_f)$ the number $\langle [\phi_1],\dots,[\phi_k] \rangle^\Phi$, such that the generating function $\F^\Phi$ of all such numbers is a solution to the WDVV equation.

The starting question of the research performed in this paper was to find the trivialization such that $\F^\Phi$ coincides with $\F^\omega$.

\subsection{In this paper}
We prove that the BV--algebra of a quasihomogeneous singularity trivializes. We construct a specific trivialization $\Phi^\omega$ of the BV--operator associated to every good basis $\omega_1,\dots,\omega_\mu$ and show that the Dubrovin connection $\nabla^\omega$ of $M_f^\omega$ is an essential lift of the connection $\nabla^\PV$ (Proposition~\ref{prop: connections match} in the text).

In Theorem~\ref{theorem: primitive form} we give a recursive formula computing the primitive form associated to a given good basis.

Finally, in Theorem~\ref{theorem: R--matrix} we give a recursive formula for the R--matrix of the Dubrovin--Frobenius manifold $M^\omega_f$ using the trivialization $\Phi^\omega$ and the computations in $\PV$. This shows that the BV--algebra generating function computed via the BV--operator trivialization $R\cdot \Phi^\omega$ coincides with the Saito theory potential $\F^\omega$.

\subsection{Acknowledgements}
The research leading to these results has received funding from the Basic Research Program at the National Research University Higher School of Economics.

The author is grateful to Sergey Shadrin for fruitful discussions.

\section{Brieskorn lattice of a singularity}\label{section: Saito--Frobenius manifold}
In this section we give a fast track through the construction of a Dubrovin--Frobenius manifold on the singularity unfolding space. This incorporates work of many people, especially K. Saito and M. Saito, to collect together the data fitting B. Dubrovin's theory. We do not aim to give a full account on the Saito--Frobenius theory, mostly highlighting only those definitions and conventions that will be needed later on. For a full presentation see \cite{H}.

\subsection{Jacobian algebra}
In what follows we will work with quasihomogeneous isolated singularities defined by a polynomial $f \in \CC[x_1,\dots,x_N]$.
Namely, we will assume that there is a set of numbers $q_1,\dots,q_N$, $q_k \in \QQ \cap (0,\frac{1}{2}]$, such that $f$ is a linear combination of monomials $x_1^{a_1}\cdots x_N^{a_N}$ subject to $a_1 q_1 + \dots + a_N q_N = 1$, and $\bx = 0$ is the only point in $\CC^N$ such that $\frac{\p f}{\p x_1} = \dots = \frac{\p f}{\p x_N} = 0$.

Associated to $f$ are the algebra $\Jac(f)$ and the $\dim(\Jac(f))$--dimensional vector space $\Omega_f$:
\[
 \Jac(f) := \CC[x_1,\dots,x_N] / (\frac{\p f}{\p x_1},\dots \frac{\p f}{\p x_N}), \quad
 \Omega_f := \Omega^{N}_{\CC^N} / df \wedge \Omega^{N-1}_{\CC^N}.
\]
The property of $\Jac(f)$ to be finite--dimensional is equivalent to the property of $f$ defining an isolated singularity and is equivalent to the fact that the partial derivatives $(\frac{\p f}{\p x_1}, \dots, \frac{\p f}{\p x_N})$ form a regular sequence. This last algebraic property will be of great importance in this text.

It is obvious that there is a vector space isomorphism $\Jac(f) \cong \Omega_f$ by $[\phi(\bx)] \mapsto [\phi(\bx) d^N\bx]$ where $d^N\bx := dx_1\wedge \dots \wedge dx_N$. An important advantage of $\Omega_f$ is that it is endowed with the bilinear form
\[
 \eta([\phi_1 d^N\bx],[\phi_2 d^N\bx]) := \frac{1}{(2 \pi \sqrt{-1})^N} \int_\Gamma \frac{\phi_1\phi_2 d^N\bx}{\frac{\p f}{\p x_1} \cdots \frac{\p f}{\p x_N}}
\]
where the integration is taken over the cycle $\Gamma := \lbrace |\frac{\p f}{\p x_1}| = \eps, \dots, |\frac{\p f}{\p x_N}| = \eps \rbrace $ for some small $\eps$.

The bilinear form is non--degenerate and well known under the name of the Poincar\'e \textit{residue pairing}. The isomorphism above allows one to lift it to $\Jac(f)$. We will denote by $\eta_f: \Jac(f) \otimes \Jac(f) \to \CC$ this lifted residue pairing. It follows immediately from the construction that
\[
 \eta_f([\phi_1],[\phi_2] \circ [\phi_3]) = \eta_f([\phi_1] \circ [\phi_2], [\phi_3])
\]
for any $[\phi_\bullet] \in \Jac(f)$ and $\circ$ being the product of $\Jac(f)$. The equality above is called the \textit{Frobenius algebra} property.

To the end of this text, fix $\mu := \dim \Jac(f)$, known under the name of \textit{Milnor number} of $f$. The algebra $\Jac(f)$ is also called the \textit{Jacobian algebra}.

\subsection{Weight and Hessian}\label{subsection: grading at the origin}
Introduce the weight function $\wt: \CC^N \to \QQ$. Set $\wt(x_k) := q_k$. Then $\wt(f) = 1$ and the weight function descends to both $\Jac(f)$ and $\Omega_f$.

Consider the polynomial $\hess(f) := \det \left( \frac{\p^2 f}{\p x_i \p x_j} \right)_{i,j=1}^N$, called the \textit{Hessian polynomial} of $f$. Then
\[
    \wt(\hess(f)) = \sum_{k=1}^N (1 - 2q_k).
\]
This polynomial is very important for $\Jac(f)$ due to the following two facts: 1. $\wt [\hess(f)] \ge \wt [\phi]$ for any homogeneous $[\phi] \in \Jac(f)$, the equality being reached if and only if $[\phi] = c [\hess(f)]$; 2. $\eta_f([1],[\hess(f)]) = \mu$ and $\eta_f([1], [a]) = 0$ for any $[a]$ belonging to the complement of $\CC \langle [\hess(f)] \rangle$ in $\Jac(f)$.

\subsection{Brieskorn lattice}
Next to $\Omega_f$, another essential object associated to the polynomial $f$ is the $\CC[z]$--module
\[
    \H^{(0)}_f := \Omega^N_{\CC^N}[z] / (zd + df\wedge ) \Omega^{N-1}_{\CC^N}
\]
called the \textit{Brieskorn lattice of $f$}. It is a $\CC[z]$--module of rank $\mu$. A basis can be taken to be $\lbrace [\phi_1(\bx) d^N \bx],\dots, [\phi_\mu(\bx) d^N \bx] \rbrace$ where the polynomials $\phi_\bullet \in \CC[\bx]$ are taken to be such that their classes generate $\Jac(f)$ as a $\CC$--vector space.

Let $\CC((z^{-1}))$ stand for the ring of Laurent series in $z^{-1}$. Define the completion of the Brieskorn lattice by
\[
    \H_f := \H^{(0)}_f \otimes_{\CC[z]} \CC((z^{-1})).
\]

Extend the weight function $\wt$ to $\Omega^N_{\CC^N}[z]$ and $\CC[\bx][z]$ by setting $\wt(z) := 1$. Obviously, it descends to both $\H_f^{(0)}$ and $\H_f$.

\section{BV algebras and homotopic resolution}
Let $(V,d)$  be a graded dg algebra. It is called a Batalin--Vilkovisky or just BV algebra if it is equipped with an additional operator $\Delta$, such that $\Delta^2 = 0$, $\Delta d + d \Delta = 0$, and $\Delta$ is of order two with respect to the product of $V$.

Assuming $z$ as a formal variable, let $\Phi \in \mathrm{End}(V)[[z]] = \Id + \sum_{k \ge 1} \Phi_k z^k$ be a homotopy of the complexes
\[
    (V^\bullet[[z]], d + z \Delta) \to (V^\bullet[[z]], d).
\]
Namely, $\Phi$ should satisfy
\begin{equation}\label{eq: BV operator trivialization}
    \Phi (d + z \Delta) = d \Phi.
\end{equation}
Such power series $\Phi$ were considered in \cite{KMS} and called \textit{$\Delta$--resolutions} of the BV algebra. Given a $\Delta$--resolution, Khoroshkin, Markarian and Shadrin constructed in \cite{KMS} the quasi-isomorphism of the Hycomm and BV operads. We use it here in a slightly modified manner.

\subsection{BV--algebra of the polyvector fields}\label{section: BV algebra}
Denote by $\PV$ the space of smooth polyvector fields on $\CC^N$. Namely, $\PV:= \bigoplus_{i \ge 0} \PV^{i}$ with $\PV^{i} := \wedge^i \T_{\CC^N}$. For $\theta_k := \frac{\p}{\p x_k}$ being the basis vectors of $\T_{\CC^N}$, an element of $\PV^{i}$ is a linear combination with complex coefficients of
\[
    \alpha(x) \theta_{r_1} \wedge \dots \wedge \theta_{r_i}
\]
with $1 \le r_\bullet \le N$ and $\alpha$ being a holomorphic function on $\CC^N$. The index value $i=0$ is allowed, in which case we will write the unit $1$ instead of the $0$--length wedge product.

Assign to $A \in  \PV^{i}$ the grading $i$ and denote $|A|:= i \mod 2$. Then the wedge product gives a $\ZZ/2\ZZ$--graded product structure on $\PV$ with the unit $1 \in \PV^{0}$.

The space $\PV$ has natural operators $ \theta_p \wedge: \PV^{i} \to \PV^{i+1}$  for any $p$. In what follows we are going to use the operator $\dfrac{\p}{\p \theta_p}: \PV^{i} \to \PV^{i-1}$ acting trivially on the $\alpha(x)$ multiple above and satisfying
\[
    \frac{\p}{\p \theta_a} (\theta_b\wedge) + (\theta_b\wedge) \frac{\p}{\p \theta_a} = \delta_{a,b}.
\]
%

Once and for all, fix a holomorphic nowhere vanishing function $\lambda$ and a holomorphic volume form
\[
    \Omega_\lambda := \frac{1}{\lambda} dx_1\wedge \dots \wedge dx_N.
\]
The contraction operator $\vdash \Omega_\lambda$ provides an isomorphism $\PV^{i} \cong \A^{N-i}$ to the space of ${(N-i,0)}$--forms on $\CC^N$. Denote by $\p$ the operator on $\PV$ lifted by this isomorphism from the Hodge theory holomorphic differential. We have
\begin{align*}
    \p \left( \alpha(x) \bigwedge_{j \in J} \theta_{j} \right)  &=
    \sum_{k=1}^N \lambda \frac{\p }{\p x_k} \left( \frac{\alpha(x)}{\lambda} \right)  \frac{\p}{\p \theta_k} \left( \bigwedge_{j \in J} \theta_{j} \right).
\end{align*}
One notes immediately that $\p^2 = 0$.

The operator $\p$ does not satisfy the Leibniz rule. Instead, $\p$ is of second order with respect to the product.
The BV operator $\p$ defines a non--trivial bracket
\[
 \lbrace A, B \rbrace := (-1)^{|A|} \p(A B) - (-1)^{|A|} \p(A) B - A \p(B)
\]
satisfying the Jacobi identity. In particular, for any polynomial $f \in \CC[x_1,\dots,x_N]$ consider the operator $\lbrace f, - \rbrace : \PV \to \PV$ taking the bracket with $f(x) \cdot 1 \in \PV^{0}$. It acts by
\begin{equation}\label{eq: bracket with f}
 \left\lbrace f, \alpha(x) \bigwedge_{j \in J} \theta_{j} \right\rbrace
 = \sum_{k=1}^N \frac{\p f}{\p x_k} \alpha(x) \frac{\p}{\p \theta_k} ( \bigwedge_{j \in J} \theta_{j} ).
\end{equation}

\subsection{BV--algebra of a hypersurface singularity}\label{section: BV of a singulariy}
Let $f \in \CC[x_1,\dots,x_N]$ define a quasihomogeneous isolated singularity at the origin $\bx = 0 \in \CC^N$.

Following \cite{LLS}, associate to it the operator $\bar\p_f : \PV \to \PV$ by
\[
    \bar \p_f := \lbrace f, - \rbrace.
\]

We have $(\bar \p_f)^2 = \lbrace f, \lbrace f, - \rbrace \rbrace = 0$.
Similarly $\bar\p_f \p + \p \bar\p_f = \lbrace f, - \rbrace \p + \p \lbrace f, - \rbrace= 0$ because $\p / \p \theta_\bullet$ anticommute. In particular, the new differential $\bar\p_f$ provides a Batalin--Vilkovisky algebra structure $(\PV, \bar \p_f, \p)$.

It is not hard to see (cf. Lemma 2.2 of \cite{LLS}) that  $H^k(\PV, \bar \p_f) = 0$ unless $k = 0$, and there is an algebra isomorphism
\begin{align}\label{eq: jac as cohomology}
H^\ast(\PV, \bar \p_f) & \cong \Jac(f),
\\
\notag
[\phi \cdot 1 ] & \mapsto [\phi].
\end{align}

Similarly we have the isomorphism
\begin{align}\label{eq: preBL as cohomology}
    H^\ast(\PV [z], \bar \p_f + z\p) & \cong \H_f^{(0)},
    \\
\notag
    [\phi(x,z) \cdot 1] & \mapsto [\phi(x,z) d^N\bx].
\end{align}
It shows that the Brieskorn lattice can be studied with the help of the BV algebra $(\PV[z], \overline{\p}_f, \p)$. However, note that the latter has a big advantage --- its product descends to a non--trivial product in cohomology.

\subsection{Topological trivialization of $(\PV,\overline \p_f, \p)$}\label{section: topological trivialization}
Denote by $\Phi^{top}$ the special trivialization of the BV--operator $\p$, inducing a trivial action on the cohomology w.r.t. $\overline \p_f$. It should satisfy
\[
    \Phi^{top} (\overline \p_f ) = (\overline \p_f + z \p) \Phi^{top}
\]
as in Eq.~\eqref{eq: BV operator trivialization}. The name \textit{topological} will be justified later.

Set
\begin{equation}
    \Phi^{top}\left(g(\bx) \cdot \bigwedge_{j \in J} \theta_{j} \right)
    := \widetilde \Phi^{top}\left(g(\bx) \right) \cdot \bigwedge_{j \in J} \theta_{j},
\end{equation}
for a $\CC$--linear map $\widetilde \Phi^{top}: \CC[\bx] \to \CC[\bx] \otimes \CC[z]$ that we define below.

The trivialization property is then equivalent to the following equality:
\[
    \sum_{k=1}^N \Phi^{top} \frac{\p f}{\p x_k} \frac{\p}{\p \theta_k} =
    \sum_{k=1}^N \left(\frac{\p f}{\p x_k} + z \frac{\p}{\p x_k} \right) \frac{\p}{\p \theta_k} \Phi^{top}.
\]

Consider $\CC[\bx]$ as an infinite--dimensional $\CC$--vector space. Let $\phi_1,\dots,\phi_\mu$ be polynomials whose classes generate $\Jac(f)$.

The polynomials $\prod_{k=1}^N \left(\frac{\p f}{\p x_k} \right)^{p_k} \phi_\alpha$ with all possible $\alpha$ and natural $p_1,\dots,p_N$ generate $\CC[\bx]$ as a vector space. The set of all such polynomials is independent of relations because $(\frac{\p f}{\p x_1},\dots,\frac{\p f}{\p x_N})$ is a regular sequence. Hence we have constructed a special basis of $\CC[\bx]$.

Because $f$ is quasihomogeneous with $\wt f = 1$, the partial derivatives $\frac{\p f}{\p x_k}$ are quasihomogeneous with $\wt (\frac{\p f}{\p x_k}) = 1 - q_k > 0$; we may assume $\phi_\bullet$ to be quasihomogeneous too. The weight of any basis element $\prod_{k=1}^N \left(\frac{\p f}{\p x_k} \right)^{p_k} \phi_\alpha$, computed by the function $\wt$, is positive unless $r_1=\dots = r_N = 0$ and $\phi_\alpha=1$.
Then every quasihomogeneous polynomial $a(x) \in \CC[\bx]$ is written as a finite $\CC$--linear combination of the special basis elements.


Set $\widetilde \Phi^{top} (\phi_\alpha) = \phi_\alpha$ and more generally
\begin{align}
    & \widetilde \Phi^{top} \left( \prod_{k=1}^N \left(\frac{\p f}{\p x_k} \right)^{p_k} \phi_\alpha \right) := \prod_{k=1}^N \left( \frac{\p f}{\p x_k} + z \frac{\p}{\p x_k} \right)^{p_k} \ \phi_\alpha.
\end{align}
The expression on the RHS is well-defined because the operators $\frac{\p f}{\p x_k} + z \frac{\p}{\p x_k}$ commute with each other for different values of $k$, and no order in the product is important.

Note that the map $\widetilde \Phi^{top}$ is quasihomogeneous because the weight of $f$ and $z$ are both equal to $1$. In particular, $\wt (\widetilde \Phi^{top}) = 0$.

Let $\Phi^{top} = \sum_{k \ge 0} \Phi^{top}_k z^k$ with $\Phi^{top}_k \in \mathrm{End}(\PV)$. Then $\Phi^{top}_0 = \mathrm{Id}$ and $\Phi^{top}$ is invertible.
%
%
%
%

\subsection{Computational aspects}\label{section: computational aspects}
The trivialization $\Phi^{top}$ is easy to compute. Indeed it is defined explicitly in the $\CC[\bx]$---basis $\prod_{k=1}^N (\frac{\p f}{\p x_k})^{r_k} \phi_\alpha$. In order to find the trivialization in the monomial basis $\prod_{k=1}^N x_k^{m_k}$, one needs to find the transition matrix between the two bases.
This is done via the inverse map $(\Phi^{top})^{-1}$. To find it one can use the formal power series expansion $(\Phi^{top})^{-1} = \sum_{m \ge 0} (-1)^m \left(\sum_{k \ge 1}\Phi_k^{top} z^k \right)^m $ that is again not hard to compute.

For $f = x_ 1^3 + x_ 2^3 x_ 1$, take $\{ \phi_1,\dots, \phi_7 \} = \left\{1, x_ 1, x_ 1^2, x_ 2, x_ 1 x_ 2, x_ 1^2 x_ 2, x_ 2^2 \right\}$.
\begin{align*}
    & x_ 2^3 = \frac{\p f}{\p x_1}\phi _ 1 - 3\phi _ 3, \quad x_ 1 x_ 2^2 = \frac{\p f}{\p x_2} \frac{\phi _ 1} {3}, \quad x_ 1^3 = \frac{\p f}{\p x_1} \frac {\phi _ 2} {3} - \frac{\p f}{\p x_2} \frac {\phi _ 4} {9},
    \\
    & x_ 2^4 =  \frac{\p f}{\p x_1}\phi _ 4 - 3\phi _ 6, \quad x_ 1 x_ 2^3 = \frac{\p f}{\p x_2} \frac{\phi_4} {3}, \quad x_ 1^2 x_ 2^2 = \frac{\p f}{\p x_2} \frac{\phi _ 2} {3}, \quad x_ 1^3 x_ 2 = \frac{\p f}{\p x_1} \frac{\phi _ 5} {3} - \frac{\p f}{\p x_2} \frac{\phi_7} {9},
    \\
    & \qquad x_ 1^4 = \frac{\p f}{\p x_1} \frac{\phi_3} {3} - \frac{\p f}{\p x_2} \frac{\phi _ 5} {9}.
\end{align*}
Then $\Phi^{top}$ acts by
\begin{align*}
 & x_ 2^3 \mapsto  x_ 2^3, \quad x_ 1 x_ 2^2 \mapsto  x_ 1 x_ 2^2, \quad x_ 1^3 \mapsto  x_ 1^3 + \frac {2 z} {9},
 \\
 & x_ 2^4 \mapsto  x_ 2^4, \quad x_ 1 x_ 2^3 \mapsto  x_ 1 x_ 2^3 + \frac {z} {3}, \quad x_ 1^2 x_ 2^2 \mapsto  x_ 1^2 x_ 2^2,
 \quad x_ 1^3 x_ 2 \mapsto  x_ 1^3x_ 2  + \frac {z x_ 2} {9}, \quad x_ 1^4 \mapsto  x_ 1^4 + \frac {5 z x_ 1} {9}.
\end{align*}

Denote by $[\phi ]_f$ the class of $\phi$ in $\Jac(f)$. Composed with the isomorphism of Eq.~\eqref{eq: preBL as cohomology}, the inverse $(\Phi^{top})^{-1}$ gives the following map $\H_f^{(0)} \to \Jac(f)[z]$:
\begin{align*}
    & 0 = [x_ 2^3 d^2\bx] \mapsto [x_ 2^3 ]_{f} = 0, \quad  0 = [x_1x_ 2^2 d^2\bx] \mapsto [x_1x_ 2^2]_{f} = 0, \quad  [x_ 1^3 d^2\bx] \mapsto  [x_ 1^3]_f - \frac {2 z} {9}[1]_f = -\frac {2 z} {9}[1]_f,
    \\
    & 0 = [x_ 2^4 d^2\bx] \mapsto [x_ 2^4]_f = 0, \quad
    [x_ 1 x_ 2^3 d^2\bx] \mapsto [x_ 1 x_ 2^3]_f - \frac{z}{3}[1]_f = - \frac{z}{3}[1]_f, \quad
    [x_ 1^2 x_ 2^2 d^2\bx] \mapsto  [x_ 1^2 x_ 2^2]_f = 0,
    \\
    & [x_ 1^3 x_ 2 d^2\bx] \mapsto  [x_1^3 x_2 ]_f - \frac {z} {9} [x_ 2]_f = - \frac {z} {9} [x_ 2]_f, \quad
    [x_ 1^4 d^2 \bx] \mapsto  [x_ 1^4]_f - \frac {5 z} {9}[x_ 1]_f =  -\frac {5 z} {9}[x_ 1]_f.
\end{align*}


\section{Dubrovin--Frobenius manifolds}

Investigation of the Dubrovin--Frobenius manifolds was sparked by the work of B. Dubrovin. In this text we only consider them in the application to Saito theory.

Let $M=(M,\O_{M})$ be a connected complex manifold of dimension $\mu$. Denote by $\T_{M}$ its holomorphic tangent sheaf.
A {\it Dubrovin--Frobenius manifold structure on M} (cf. \cite{D1,D2}) is a tuple $(\eta, \circ , e,E)$, where $\eta$ is a non--degenerate $\O_{M}$--symmetric bilinear form on $\T_{M}$, $\circ $ is an $\O_{M}$--bilinear product on $\T_{M}$, defining an associative and commutative $\O_{M}$--algebra structure with a unit $e$, and $E$ is a holomorphic vector field on $M$, called the Euler vector field, which are subject to the following axioms:
\begin{enumerate}
\item $ \eta(X\circ Y,Z)=\eta(X,Y\circ Z)$ for all $X,Y,Z \in\T_M$.
\item The {\rm Levi}--{\rm Civita} connection $\ns$ of $\eta$ is flat.
\item The tensor $C:\T_M\otimes_{\O_M}\T_M\to \T_M$  defined by
$C_X Y := X\circ Y$, for all $X,Y\in\T_M$ is flat with respect to $\ns$.
\item The unit element $e$ of the $\circ $-algebra is a
$\ns$-flat holomorphic vector field.

\item There is $d\in\CC$, such that the metric $\eta$ and the product $\circ$ are homogeneous of degree $2-d$ and $1$ respectively with respect to the Lie derivative
$Lie_{E}$ of the {\rm Euler} vector field $E$: that is,
\begin{equation*}
Lie_E(\eta)=(2-d)\eta,\quad Lie_E(\circ)=\circ.
\end{equation*}
\end{enumerate}

The structure of a Dubrovin--Frobenius manifold can be locally described by an analytic function $\F$, called its \textit{potential}. Namely, let $\mu = \dim M$ and $t_1,\dots,t_\mu$ be flat coordinates of the Levi-Civita connection above. At a point $p \in M$, consider $T_pM = \CC \langle \frac{\p}{\p t_1} ,\dots,\frac{\p}{\p t_\mu}\rangle$. Assume in addition that $t_1$ is such that $e = \frac{\p}{\p t_1}$, $\eta_{ij} = \eta(\frac{\p}{\p t_i},\frac{\p}{\p t_j})$ -- the components of $\eta$ in the basis fixed, and $\eta^{ij}$ being components of $\eta^{-1}$. Then there is a function $\F=\F(t_1,\dots,t_\mu)$ such that
\begin{align}
    & \frac{\p}{\p t_i} \circ \frac{\p}{\p t_j} = \sum_{k,l=1}^\mu \frac{\p^3 \F}{\p t_i \p t_j \p t_l} \eta^{lk} \frac{\p}{\p t_k},
    \label{eq: FDMF-1}
    \\
    & \eta_{ij} = \frac{\p^3 \F}{\p t_1 \p t_i \p t_j}, \qquad E \cdot \F = (3 - d) \F + \text{terms quadratic in $t_\bullet$}. \label{eq: FDMF-2}
\end{align}
Locally the potential $\F$ fully encodes the data of a Frobenius manifold $M$.

\subsection{Formal Dubrovin--Frobenius manifold}\label{subsection: formal Dubrovin--Frobenius manifold}
Consider a formal power series $\F \in \CC[[t_1,\dots,t_\mu]]$ that satisfies Eq.~\eqref{eq: FDMF-2} for some $E$, $\eta$ and $d$. Introduce the product $\circ$ on the formal vectors $\frac{\p}{\p t_1},\dots, \frac{\p}{\p t_\mu}$ by the formula \eqref{eq: FDMF-1}. Let this product be commutative and associative.

Then $\F$ is said to define a \textit{formal} Dubrovin--Frobenius manifold. It would become a true Dubrovin--Frobenius manifold if one associates a complex manifold $M$ to $\F$, such that $\F \in \O_M$.

Note that the associativity condition on the product becomes a system of PDEs on the function $\F$ called the \textit{WDVV equation}.

In particular, the genus zero generating functions of Gromov--Witten theory or any cohomological field theory define some formal Dubrovin--Frobenius manifolds. This follows from the topology of the moduli space of genus zero curves, implying that the corresponding generating functions satisfy the WDVV equation.

Examples of formal Dubrovin--Frobenius manifolds are given by topological cohomological field theories. Let $(A,\circ,\eta,e)$ be a $\mu$--dimensional Frobenius algebra; then its topological cohomological field theory has the genus zero potential
\[
    \F^{top}(t_1,\dots,t_\mu)
    := \frac{1}{3!} \sum_{\alpha_1,\alpha_2,\alpha_3 = 1}^\mu \eta (\phi_{\alpha_1} \circ \phi_{\alpha_2},\phi_{\alpha_3}) t_{\alpha_1}t_{\alpha_2}t_{\alpha_3}.
\]

Even though formal Dubrovin--Frobenius manifolds lack such an important property as analyticity, they enjoy the following additional advantage. A. Givental introduced in \cite{G01a,G04} a group action on the space of formal Dubrovin--Frobenius manifolds. This group action will play an important role later on.

\subsection{Dubrovin connection}\label{subsection: Dubrovin connection}
Associate to any Dubrovin--Frobenius manifold the connection on $M \times \CC^\ast_z$ given by
\begin{align}\label{eq: Dubrovin connection}
    \widetilde \nabla_X := \ns_X + \frac{1}{z} C_X, \quad \widetilde \nabla_{\frac{d}{dz}} := \frac{d}{dz} + \frac{1}{z} \left( B_0 + \frac{B_\infty}{z} \right)
\end{align}
for $C_X(Y) := X \circ Y$, $B_0(Y) := E \circ Y$, and the diagonal grading operator $B_\infty(Y) := \dfrac{2-d}{2}Y - \nabla_Y E$.

This new connection is flat. It provides an important piece of data of a Dubrovin--Frobenius manifold. In particular, its flat sections reconstruct the product uniquely.

Given a manifold $M$ with a flat pairing $\eta$ and a fiberwise product $\circ$ on $\T_M$ satisfying the Frobenius algebra property, if $\widetilde \nabla$ is flat, then $M$ is a Dubrovin--Frobenius manifold (cf. Proposition 2.1 in \cite{D2} and \cite{Sa}). This gives one of the ways to construct Dubrovin--Frobenius manifolds.

\subsection{Givental's group action}\label{section: givental's action}
Let $\eta$ be a $\mu \times \mu$ matrix of a constant non--degenerate bilinear form. The space of all matrix--valued formal power series $R = \Id + \sum_{k \ge 1} R_k z^k$ with $R_k \in \mathrm{Mat}(\CC,\mu)$, satisfying the condition $R(z)\eta^{-1}R^T(-z) = \eta^{-1}$, is called {\it Givental's group}.

By A. Givental, this group acts on the space of formal Dubrovin--Frobenius manifolds of dimension $\mu$ with metric $\eta$ (cf. \cite{G01a}). This group action is now known under the name of Givental's group action\footnote{also called the upper-triangular group action.}.

Givental's action has appeared to be very important in different applications. It can be introduced via the graph summation formula \cite{DbSS} or by the action of a differential operator \cite{G01a, L08}. However, the particular formulae will not play any role in our text and are therefore skipped.

\subsection{From BV/$\Delta$ to a Dubrovin--Frobenius manifold}
In \cite{KMS} Khoroshkin, Markarian and Shadrin constructed the quasi-isomorphism of operads $\theta: \mathrm{Hycomm} \to \mathrm{BV}/\Delta$. To every BV--algebra $(V,d,\Delta)$ with trivialized BV--operator they associate the map
\[
    \theta_n: H^\ast(V,d) ^{\otimes n} \to H^\ast(V,d).
\]
In case $H^\ast(V,d)$ is a Frobenius algebra with pairing $\eta$ and unit $1$, the correlators
\[
    \langle a_1,\dots,a_n \rangle^\Phi := \eta\left(1, \theta_n(a_1,\dots,a_n) \right)
\]
are the correlators of some genus zero cohomological field theory. This construction strongly depends on the choice of the BV--operator trivialization, and this is the reason why we put superscript $\Phi$ on the correlators.

In particular, if $H^\ast(V,d) \cong \CC \langle \phi_1,\dots, \phi_\mu \rangle$, then the generating function
\[
    \F^\Phi(t_1,\dots,t_\mu) := \sum_{n \ge 3} \frac{1}{n!} \sum_{\alpha_1,\dots,\alpha_n} \langle \phi_{\alpha_1},\dots,\phi_{\alpha_n} \rangle^\Phi t_{\alpha_1}\cdots t_{\alpha_n} \in \CC[[t_1,\dots,t_\mu]]
\]
is a solution to the WDVV equation and defines a formal Dubrovin--Frobenius manifold.

\subsubsection{Topological trivialization $\Phi^{top}$}
For the BV--algebra $(\PV,\overline{\p}_f,\p)$, the trivialization $\Phi^{top}$ was called topological due to the following reason:
\[
    \Phi_k^{top} ([a]) \equiv 0 \in H^\ast(\PV,\overline\p_f + z \p), \quad \forall k \ge 1 \ \text{and} \ [a] \in H^\ast(\PV,\overline{\p}_f).
\]
Then the construction of Khoroshkin, Markarian and Shadrin gives
\[
    \F^{\Phi^{top}}(t_1,\dots,t_\mu) = \frac{1}{3!} \sum_{\alpha_1,\alpha_2,\alpha_3} \langle \phi_{\alpha_1},\phi_{\alpha_2},\phi_{\alpha_3} \rangle^{\Phi^{top}} t_{\alpha_1}t_{\alpha_2}t_{\alpha_3}
    = \frac{1}{3!} \sum_{\alpha_1,\alpha_2,\alpha_3} \eta (\phi_{\alpha_1} \circ \phi_{\alpha_2},\phi_{\alpha_3}) t_{\alpha_1}t_{\alpha_2}t_{\alpha_3}.
\]
This is exactly the genus zero potential of the topological CohFT associated with $\Jac(f)$.

We are going to use special trivializations of $\p$ in order to get the Saito theory Dubrovin--Frobenius manifold later on.

\subsection{Different trivializations}\label{section: different trivializations}
Another major result of \cite{KMS} is the following. Fix the BV--algebra $(A,d,\Delta)$, and let $\Phi$ be a trivialization of the BV--operator. Assume $R = \Id + \sum_{p \ge 1} z^p R_p \in \End(V)[[z]]$ commutes with $d$. Then $R\Phi$ is another trivialization of the BV--operator of the same BV--algebra.

Let $H^\ast(V,d)$ be endowed with the pairing and unit. Then we have two sets of correlators
\[
 \langle a_1,\dots,a_n \rangle^{\Phi} \quad \text{and} \quad \langle a_1,\dots,a_n \rangle^{R\Phi}
\]
giving, a priori, two different generating functions $\F^{\Phi}$ and $\F^{R\Phi}$. Khoroshkin, Markarian and Shadrin prove:

\begin{theorem}[Theorem 6.3 in \cite{KMS}]\label{theorem: KMS different trivializations}
    If $R$, assumed as an operator on $H^\ast(V,d)$, satisfies the condition $R(z)\eta^{-1}R^T(-z) = \eta^{-1}$, then the generating functions $\F^{\Phi}$ and $\F^{R\Phi}$ are connected by the Givental action of $R$ (see Section~\ref{section: givental's action}).
\end{theorem}


\section{Brieskorn lattice and BV algebra of the unfolding}\label{section: unfolding}
Given a quasihomogeneous singularity, Saito's theory allows one to construct a Dubrovin--Frobenius manifold structure on the unfolding space of the singularity. This construction is not unique, depending on the additional choice --- primitive form or opposite subspace. Main references are \cite{ST, H} and the citations therein.

\subsection{Unfolding of a singularity}

Consider an \textit{unfolding} of $f$. Namely, the function $F : Z \to \CC$, where $Z = \CC^N \times \calS$, for some open neighbourhood of the origin $\calS \subset \CC^\mu$ with coordinates $s_\bullet$, defined by
\[
    F(\bx,\bs) = f(x) + \sum_{k = 1}^\mu \phi_k(x) s_k.
\]
Here the polynomials $\phi_\bullet \in \CC[\bx]$ are taken to be such that their classes generate $\Jac(f)$. We also assume these polynomials to be homogeneous with respect to the function $\wt$ introduced in previous sections.

Let $p: Z \to \calS$ be the projection on the second factor. Consider the critical sheaf
\[
    \O_\C := \O_{Z} / \left( \p_{x_1}F, \dots, \p_{x_N} F \right).
\]
Then $p_\ast \O_\C$ is an $\O_\S$--module of rank $\mu$ with the essential product structure.

Let $(\Omega^\bullet_{Z/\calS},d_{Z/\calS})$ stand for the de Rham complex relative to $p$. Consider the direct image sheaf
\[
    \RR^k p_\ast \left( \Omega^\bullet_{Z/\calS} [z],z d_{Z/\calS} + dF \wedge \right).
\]
It vanishes for $k \neq N$. And when $k=N$, it is isomorphic to the sheaf
\[
    \H_F^{(0)} := \Omega^N_{Z/\calS} [z] / (z d_{Z/\calS} + dF \wedge ) \Omega^{N-1}_{Z/\calS} [z].
\]
The fiber of it over a point $\bs \in \calS$ is given by $\H^{(0)}_g$ with $g: \CC^N \to \CC$ given by $g(x) = F(x,s)$. This sheaf is a locally free $\O_\calS[z]$--module of rank $\mu$ (cf. Proposition 3.5 of \cite{IMRS}).

\subsubsection{Gauss--Manin connection}\label{section: Gauss-Manin connection}
$\H_F$ is endowed with the Gauss--Manin connection $\nabla^\GM$ defined as follows. Let $v \in \T_\calS$, $\phi \in \O_\calS[z]$, and $d^N\bx = dx_1\wedge \dots \wedge dx_N$. Denote by $v(\phi)$ the directional derivative along $v$ and by $[\phi d^N\bx]$ the class of $\phi d^N\bx$ in $\H_F$.
\begin{align}
    & \nabla_v^\GM [\phi d^N\bx] := \left[ (v (\phi) + z^{-1} \phi \cdot v(F)) d^N\bx \right],
    \\
    & \nabla_{\frac{\p}{\p z}}^\GM [\phi d^N\bx] := \left[ \left(\frac{\p \phi}{\p z} - z^{-2} \phi \cdot F - \frac{N}{2} z^{-1} \phi \right) d^N\bx \right].
\end{align}
The Gauss-Manin connection satisfies the Leibniz rule and is flat.

This connection is not of the form of Dubrovin connection~\eqref{eq: Dubrovin connection}. Some additional work has to be done in order to put it in the right form. There are two ways to do that. The first one uses the period mapping defined by the primitive form, and the second one is essentially in the form of S. Barannikov's semi-infinite variations of Hodge structures.

Both approaches make use of the higher residue pairing and good bases.


\subsubsection{Higher residue pairing}\label{section: higher residue pairing}
K. Saito introduced the pairing (cf. \cite{SK2})
\[
   K_F: \H_F^{(0)}\otimes_{\O_\S} \H_F^{(0)}\to \O_\S[z]
\]
called the \textit{higher residue pairing}, uniquely defined by the following properties.
Let
\[
 K_F(\omega_1, \omega_2) = \sum_{p \ge 0} z^p K_F^{(p)}(\omega_1,\omega_2)
\]
then we have
\begin{enumerate}

\item $K_F^{(p)}(\omega_1, \omega_2) = (-1)^p K_F^{(p)}(\omega_2,\omega_1)$.

\item $K_F(z\omega_1, \omega_2) = -K_F(\omega_1, z \omega_2) = zK_F(\omega_1,\omega_2)$.

\item $K_F^{(0)}$ defines the pairing
$$
    \H_F^{(0)}/z  \H_F^{(0)}\otimes  \H_F^{(0)}/z  \H_F^{(0)} \to \CC, \quad \omega_1\otimes \omega_2\mapsto K_F^{(0)}(\omega_1,\omega_2)
$$
which coincides with the residue pairing $\eta$.

\item $K_F$ is flat with respect to the Gauss--Manin connection:
\[
    \xi \cdot K_F (\omega_1,\omega_2) = K_F ( \nabla_\xi^\GM \omega_1,\omega_2) - K_F(\omega_1,\nabla_\xi^\GM \omega_2)
\]
for $\xi = z^2 \frac{\p}{\p z}$ or $\xi \in z \T_\S$.
\end{enumerate}
This is a theorem of K.~Saito that these properties fix the higher residue pairing completely.

The higher residue pairing extends to a pairing $\omega_F : \H_F \otimes \H_F \to \O_\S$ by
\[
   \omega_F(\omega_1, \omega_2):= \res_{z=0} K_F(\omega_1, \omega_2) dz.
\]

The Brieskorn lattice $\H_f^{(0)}$ is endowed with the pairing $K_f: \H_f^{(0)} \otimes \H_f^{(0)} \to \CC[z]$ given by $K_f := K_F \mid_{\bs = 0}$. Similarly, $\H_f$ has the pairing $\omega_f := \omega_F \mid_{\bs =0}$.

\subsubsection{Weight function}
Extend the weight function $\wt$ to $\H_F$. To do this we need to extend $\wt$ to $\O_\S$. This is done by requiring $\wt(F) = 1$. Then $\wt (s_\alpha \phi_\alpha) = \wt (s_\alpha) + \wt(\phi_\alpha) = 1$ for any $\alpha=1,\dots,\mu$, giving $\wt (s_\alpha) = 1 - \wt(\phi_\alpha)$.

For quasihomogeneous singularities, the flatness of $K_F$ with respect to the Gauss--Manin connection in the $\p / \p z$ direction is equivalent to the condition
\[
   K_F^{(p)}([\alpha d^N\bx], [\beta d^N\bx] ) = 0 \ \text{unless} \ \deg \alpha + \deg \beta = p + \deg \hess(f).
\]
for any homogeneous elements  $[\alpha d^N\bx], [\beta d^N\bx] \in \H_F^{(0)}$. Note that this condition holds for $K_F^{(0)}$ and the residue pairing.

\subsection{Good bases}
A subspace $\L \subset \H_f$ is called a \textit{good opposite subspace} if the following four conditions hold:
\begin{enumerate}
    \item $\H_f= \H_f^{(0)} \oplus \L$,
    \item the direct sum decomposition above is quasihomogeneous with respect to $\wt$,
    \item $\L$ is isotropic with respect to $\omega_f$,
    \item $z^{-1}(\L) \subseteq \L$.
\end{enumerate}

Given a good opposite subspace, denote $B := \H_f^{(0)} \cap z \L$. Then we have $B \cong \Jac(f)$, and the following are equivalent (cf. \cite[Lemma$\backslash$Definition 2.16]{LLS})
\[
    \omega_f(\L,\L) = 0 \ \Leftrightarrow \ K_f(\L,\L) \subset z^{-2}\CC[z^{-1}] \ \Leftrightarrow \ K_f(B,B) \subset \CC.
\]
In particular, $\L$ constitutes the decomposition
\begin{equation}\label{eq: Hf0 via the good basis}
    \H^{(0)}_f = B[z], \quad \L = z^{-1} B[z^{-1}].
\end{equation}
In what follows we call a basis $\lbrace \omega_\alpha(x,z)\rbrace_{\alpha=1}^\mu$ of $B$ fixed by the choice of $\L$ a \textit{good basis}.

Given a fixed good basis, choose the polynomials $\phi_\bullet$ whose classes generate $\Jac(f)$ by the equality
\begin{equation}\label{eq: omega_bullet to phi_bullet}
    \omega_a \equiv [ \phi_a d^N \bx] \quad \text{ mod } z \H_f^{(0)}.
\end{equation}
These polynomials are homogeneous due to condition (2) of the good opposite subspace. We may also assume $\phi_1 = 1$.

\subsubsection{Good basis of $\H_F^{(0)}$}

Flatness of the Gauss-Manin connection allows one to extend uniquely an opposite subspace $\L$ and a good basis $\lbrace \omega_\alpha \rbrace_{\alpha=1}^\mu$ over a contractible subspace of $\S$.
We will denote by the same letter $\omega_\bullet \in \G$ these extended good basis elements. Li--Li--Saito constructed a particular map $\E$ that realizes this extension (see Section 4.1.3 of \cite{LLS}). In particular, it follows that every good basis of $\G$ is obtained as the image of a good basis of $\H_f^{(0)}$ under $\E$.

\subsection{From Gauss-Manin to Dubrovin's connection}\label{subsection: DC from GM}
Fix a good basis $\omega_1,\dots,\omega_\mu$ of $\H_F^{(0)}$ and write the Gauss--Manin connection in it. It follows from the properties of the higher residue pairing that in this basis $\nabla^\GM$ attains the form \eqref{eq: Dubrovin connection} and we get the Dubrovin--Frobenius manifold as in Section~\ref{subsection: Dubrovin connection}.

In what follows we denote by $M^\omega_f$ this Dubrovin--Frobenius manifold and by $\nabla^{\omega}$ the Gauss--Manin connection $\nabla^\GM$ computed in the good basis.

The operator $E \circ$ of the Dubrovin--Frobenius manifold coincides with the operator of multiplication $F \cdot $ in $\H_F^{(0)} / z \H_F^{(0)}$.

The metric $\eta$ of the Dubrovin--Frobenius manifold is given by the higher residue pairing. The latter depends additionally on the variable $z$, however this dependence does not show up when evaluating $K_F$ on the good basis.

The coordinates $s_1,\dots,s_\mu$ introduced in the previous sections are generally not flat for the pairing $\eta$ and new coordinates $t_1,\dots,t_\mu$ need to be introduced. The Euler vector field in the flat coordinates assumes the form $E = \sum_{\alpha=1}^\mu \wt(s_\alpha) t_\alpha \frac{\p}{\p t_\alpha}$.

\section{Dubrovin connection via the BV algebra}
The construction of Dubrovin--Frobenius manifold $M_f^\omega$ required the choice of the basis in $\H_f^{(0)}$.
In this section we give basis--free definition of Dubrovin connection in Saito theory.

\subsection{BV algebra $(\PV,\overline \p_F, \p)$ and its trivialization}

Consider the space $\PV \otimes \O_\S$ and the operator $\bar\p_F$  on it given by
\[
    \bar \p_F := \lbrace F, - \rbrace.
\]
We have again $(\bar \p_F)^2 = 0$ and $\bar\p_F \p + \p \bar\p_F = 0$, what gives us
the BV algebra $(\PV \otimes \O_\S, \bar \p_F, \p)$.

Similarly to Section~\ref{section: BV of a singulariy} we have the isomorphisms
\begin{align}
    \Upsilon_0: H^\ast(\PV\otimes\O_\S, \bar \p_F) & \cong p_\ast \O_\C,
    \label{eq: HPV to OC iso}
    \\
    \notag
    [\phi(\bx,\bs) \cdot 1 ] & \mapsto [\phi(\bx,\bs)].
    \\
    \Upsilon_1: H^\ast(\PV \otimes \O_\S[z], \bar \p_F + z\p) & \cong \H_F^{(0)},
    \label{eq: HPV to HF iso}
    \\
\notag
    [\phi(\bx,\bs,z) \cdot 1] & \mapsto [\phi(\bx,\bs,z) d^N\bx].
\end{align}
See also Section~3 and Proposition~3.5 of \cite{LLS}.

We construct the trivialization of $(\PV \otimes \O_\S, \bar \p_F, \p)$.

By the definition of $\S$ for any fixed $\bs \in \S$ the map $g: \CC^N \to \CC$ given by $g(\bx) = F(\bx,\bs)$ defines an isolated singularity. Our construction of the topological trivialization does not work for $g$ because it is not necesserily quasihomogeneous.

However the set $(\frac{\p F}{\p x_1},\dots, \frac{\p F}{\p x_N})$ is a regular sequence for any $\bs \in \S$ and we can construct the topological trivialization of the new BV algebra following the steps of Section~\ref{section: topological trivialization}.

Indeed for $\phi_\alpha := \frac{\p F}{\p s_\alpha}$ the set $\prod_{k=1}^N \left( \frac{\p F}{\p x_k} \right)^{p_k} \phi_\alpha$ is the basis of $\CC[\bx] \otimes_\CC \O_\S$ assumed as an $\O_\S$--module. Namely, we can express every element of $\CC[\bx]$ uniquelly as the linear combination of the basis elements with the $\O_\S$--valued coefficients. This sum is finite because it holds for all $\bs$ and $\bs = 0$ in particular where we know it to be finite by Section~\ref{section: topological trivialization}.

The map $\Phi^{top}_F: \PV \otimes \O_\S[z] \to \PV \otimes \O_\S[z]$ given by
\begin{align}
    & \Phi^{top}_F\left(g(\bx,\bs) \cdot \bigwedge_{j \in J} \theta_{j} \right)
    := \widetilde \Phi^{top}_F\left(g(\bx,\bs) \right) \cdot \bigwedge_{j \in J} \theta_{j},
    \\
    & \widetilde \Phi^{top}_F \left( \prod_{k=1}^N \left(\frac{\p F}{\p x_k} \right)^{p_k} \phi_\alpha \right) := \prod_{k=1}^N \left( \frac{\p F}{\p x_k} + z \frac{\p}{\p x_k} \right)^{p_k} \ \phi_\alpha
\end{align}
trivializes $\p$ in $(\PV \otimes \O_\S, \bar \p_F, \p)$.

Decompose $\Phi^{top}_F = \sum_{k \ge 0} \Phi^{top}_{F,k} z^k$ for some $\Phi^{top}_{F,k}: \PV\otimes\O_\S \to \PV\otimes\O_\S$. It follows immediately from the construction that $\Phi^{top}_{F,0} = \mathrm{Id}$. This shows that $\Phi^{top}_F$ is invertible.

\begin{proposition}\label{prop: tautology}~
    \begin{enumerate}
     \item[(i)] The map $\Upsilon_0\left( \Phi^{top}_F \right)^{-1} \Upsilon_1^{-1}$ establishes the isomorphism $\H_F^{(0)} \to p_\ast \O_\C[z]$ mapping $[\phi_\alpha d^N\bx]$ to $[\phi_\alpha]$.
     \item[(ii)] For any $\omega \in \PV\otimes\O_\S[z]$ let $[\left(\Phi^{top}_F(\omega)\right)^{-1}] = \sum_{k=1}^\mu a_k(\bs) [\phi_k]$ in $p_\ast \O_\C[z]$. Then $[\omega] = \sum_{k=1}^\mu a_k(\bs,z) [\phi_k d^N \bx]$ in $\H_F^{(0)}$.
    \end{enumerate}
\end{proposition}
\begin{proof}
    Part (i) follows immediately becasue $\Phi^{top}_F$ is quasiisomorphism.

    To prove (ii) note that because of Eq.~\eqref{eq: HPV to HF iso} and Eq.~\eqref{eq: HPV to OC iso} this is equivalent to the same statement about $H^*(\PV\otimes\O_\S[z],\overline\p_F + z \p)$ and $H^*(\PV\otimes\O_\S[z],\overline\p_F)$.

    By part (i), the image $[\left(\Phi^{top}_F\right)^{-1}(\omega)]$ only depends on the cohomology class $[\omega]$ and not on $\omega$ itself. Let $[\left(\Phi^{top}_F\right)^{-1}(\omega)] = \sum_{k=1}^\mu a_k(\bs,z) [\phi_k]$ in $p_*\O_\S$. Then $[\omega] = \Phi^{top}_F [\left(\Phi^{top}_F\right)^{-1}(\omega)] = \sum_{k=1}^\mu a_k(\bs,z) \Phi^{top}_F ( [\phi_k]) = \sum_{k=1}^\mu a_k(\bs,z) [\phi_k]$ in $\H_F^{(0)}$.
\end{proof}

\subsubsection{Example}
For $f = x_ 1^3 + x_ 2^3 x_ 1$ as in Section~\ref{section: computational aspects} take $\{ \phi_1,\dots, \phi_7 \} = \left\{1, x_ 1, x_ 1^2, x_ 2, x_ 1 x_ 2, x_ 1^2 x_ 2, x_ 2^2 \right\}$ and $F = f + \sum_{\alpha=1}^7 s_\alpha \phi_\alpha$.
iDenote by $[\phi ]_F$ the class of $\phi$ in $p_\ast \O_C$. Composed with $\Upsilon_0$ and $\Upsilon_1$ the inverse $(\Phi_F^{top})^{-1}$ gives the following map $\H_F^{(0)} \to p_\ast \O_C[z]$
\begin{align*}
    & [x_ 2^3 d^2\bx] \mapsto [x_ 2^3 ]_F, \quad  [x_1x_ 2^2 d^2\bx] \mapsto [x_1x_ 2^2]_F, \quad  [x_ 1^3 d^2\bx] \mapsto  [x_ 1^3]_F - z \frac {2} {9}[1]_F ,
    \\
    & [x_ 2^4 d^2\bx] \mapsto [x_ 2^4]_F,
    \quad
    [x_ 1 x_ 2^3 d^2\bx] \mapsto [x_ 1 x_ 2^3]_F - \frac{z}{3}[1]_F,
    \quad
    [x_ 1^2 x_ 2^2 d^2\bx] \mapsto  [x_ 1^2 x_ 2^2]_F +  z \frac{2}{27} s_6[1]_F,
    \\
    & [x_ 1^3 x_ 2 d^2\bx] \mapsto  [x_1^3 x_2 ]_F - \frac {z} {9} [x_ 2]_F -  z \frac{10}{243} s_6^2[1]_F,
    \\
    & [x_ 1^4 d^2 \bx] \mapsto  [x_ 1^4]_F + z \frac {5} {9}[x_ 1]_F - \frac{5}{81} s_6 z [x_2[_F - z \frac{50 s_6^3 }{2187}[1]_F - z \frac{4 s_3}{27}[1]_F.
\end{align*}

\subsection{Dubrovin connection}
One notes immediately that $(\PV \otimes \O_\S, \bar \p_F, \p)$ is endowed with the connection as follows. Consider the operators on $\PV \otimes \O_\S((z^{-1}))$
\begin{align}
    & \nabla_v^\PV (a) := v(a) + z^{-1} a v(F), \quad \forall v \in \T_\S,
    \\
    & \nabla_{\frac{\p}{\p z}}^\PV (a) := \frac{\p a}{\p z} - z^{-2} a F.
\end{align}
Direct computations show that
\[
    [\nabla_v^\PV, \overline\p_F + z  \p] = 0, \quad [\nabla_{\frac{\p}{\p z}}^\PV, \overline \p_F + z \p] = z^{-1} (\overline \p_F + z \p).
\]
We conclude that $\nabla_v^\PV$ and $\nabla_{\frac{\p}{\p z}}^\PV$ are well--defined on the cohomology $H^\ast( \PV \otimes \O_\S((z^{-1})), \bar \p_F + z\p)$.

\begin{remark}
    At this point we could have chosen to speak in terms of the so-called ``z--connections''. Namely, consider the operators $z \nabla_v^\PV$ and $z^2 \nabla_{\frac{\p}{\p z}}^\PV$ acting on $H^\ast( \PV \otimes \O_\S[z], {\bar \p_F + z\p})$.
    However some choice is already done on the singularity theory side --- by extending $\H_F^{(0)}$ to $\H_F$.
\end{remark}

Fix a good basis $\omega_1,\dots,\omega_\mu$ of $\H_F^{(0)}$, and let $\phi_1,\dots,\phi_\mu$ be the $\CC[\bx]\otimes\O_\S$--elements such that
\[
 \omega_\alpha = [ \phi_\alpha d^N\bx] \ \mod z\H_F^{(0)}.
\]
Then the classes $[\phi_1 \cdot 1],\dots,[\phi_\mu \cdot 1]$ give a basis of $H^\ast(\PV \otimes \O_\S,\overline \p_F)$ as an $\O_S$--module.

Obviously, $\Upsilon_1$ extends to an isomorphism $H^\ast(\PV \otimes \O_\S(z), \bar \p_F + z\p)  \cong \H_F$ that we denote by the same letter.

\begin{proposition}\label{prop: connections match}~
\begin{enumerate}
    \item[(i)] There is a trivialization $\Phi^\omega$ of $(\PV \otimes \O_\S, \bar \p_F, \p)$ such that
        \[
            \Upsilon_1\Phi^\omega\Upsilon_0^{-1}: p_\ast\O_\C \to \H_F^{(0)} \quad \text{maps} \quad [\phi_\alpha] \mapsto \omega_\alpha.
        \]
    \item[(ii)] The Gauss--Manin connection satisfies
    \[
        \nabla^\GM = \Upsilon_1^{-1} \cdot \nabla^\PV \cdot \Upsilon_1.
    \]
    \item[(iii)] The Dubrovin connection fixed by a good basis satisfies
    \[
        \nabla^\omega = (\Phi^\omega)^{-1}\Upsilon_1^{-1} \cdot \nabla^\PV \cdot \Upsilon_1 \Phi^\omega.
    \]
\end{enumerate}

\end{proposition}
\begin{proof}
    The classes of the polynomials $\phi_1,\dots,\phi_\mu$ generate $p_\ast \O_\C$, and there is $M \in \End(\Jac(f)[z])\otimes \O_\S$ such that $\omega_i = M(\phi_i)$.

    Let $M$ act on $\PV\otimes \O_\S$ by $\prod_{k=1}^N \left(\frac{\p F}{\p x_k} \right)^{p_k} \phi_\alpha \wedge_{i\in I} \theta_i  \mapsto \prod_{k=1}^N \left(\frac{\p F}{\p x_k} \right)^{p_k} M(\phi_\alpha) \wedge_{i\in I} \theta_i$. This is a linear map commuting with $\overline{\p}_F$ and $\p$.
    Then the composition $\Phi^\omega := M \Phi^{top}$ is the trivialization claimed in (i).

    The other two statements follow immediately from (i).
\end{proof}

\[
    \begin{tikzcd}
    H^\ast(\PV\otimes \O_\S((z^{-1})),\overline{\p}_F + z \p ) \arrow[loop left, "\nabla^\PV"] \arrow{r}{\Upsilon_1} &  \H_F \arrow[loop right, "\nabla^\GM"]
    \\
    H^\ast(\PV\otimes \O_\S,\overline{\p}_F)((z^{-1})) \arrow[loop left, "\nabla^\omega"] \arrow{r}{\Upsilon_0} \arrow{u}{\Phi^\omega} &  p_\ast \O_\C \arrow{u}((z^{-1}))
    \end{tikzcd}
\]

\section{Primitive form calculus}

Another way to construct a Dubrovin--Frobenius manifold in Saito theory is via the period mapping of K. Saito (\cite{SK3}).

Fix an element $\zeta \in \H_F^{(0)}$ and consider the map $\Psi: \T_\S \to \H_F^{(0)} / z \H_F^{(0)}$ given by $\Psi(X) := z \nabla_X \zeta$. If one makes a special choice of $\zeta$, called nowadays the \textit{primitive form} of K. Saito, then $\Psi$ is an isomorphism of $\O_\S$--modules transforming the connection $\nabla^{\GM}$ to a Dubrovin connection.

The choice of primitive form is usually not unique. Its existence was proved by M. Saito \cite{SM1,SM2}. He also proved that there is a one-to-one correspondence between primitive forms and good bases of the previous sections.


\subsection{``Perturbative'' primitive form calculus}
Fix a point $s \in \calS$. It was observed in \cite{LLS} that the formal operator $e^{(F-f)/z}$ commutes two differentials:
\[
    (\overline \p_f + z \p) e^{(F-f)/z} = e^{(F-f)/z} (\overline \p_F + z \p).
\]
This is a starting point to the following recipe to compute the primitive form.

Fix an opposite subspace $\L$ of $\H_f^{(0)}$. Let $B := z\L \cap \H_f^{(0)}$. Introduce the exponential operator $e^{(F-f)/z}: B \to B((z^{-1}))[[\bs]]$ by
\[
 e^{(F-f)/z} \left( [\phi_\alpha d^\bx] \right) := \sum_{k =0}^\infty \sum_{\beta=1}^\mu \sum_{m \ge -k} h^{(k)}_{\alpha\beta,m} \frac{z^m}{k!} [\phi_\beta d^N\bx]
\]
for $h_{\alpha\beta,m}^{(k)}$ obtained by the following equality in $\H_f \otimes \O_\S$:
\[
    [z^{-k} (F-f)^k \phi_\alpha d^N\bx] = \sum_{\beta=1}^\mu \sum_{m \ge -k} h^{(k)}_{\alpha\beta,m} \frac{z^m}{k!} [\phi_\beta d^N\bx].
\]
Note that the exponential operator depends on the choice of the opposite subspace. Its components $h^{(k)}_{\alpha\beta,m}$ assume decomposing an $\H_f$ element in the basis fixed by the opposite subspace.

C. Li, S. Li and K. Saito prove:
\begin{theorem}[Theorem~4.15 of \cite{LLS} and Theorem~3.7 of \cite{LLSS}]
    There exists a unique pair $(\zeta,J)$ with $\zeta \in B[z][[\bs]]$ and $J \in [d^N\bx] + z^{-1}B[[z^{-1}]][[\bs]]$ such that
    \begin{equation}\label{eq: main Jzeta}
        J = e^{(F-f)/z}(\zeta).
    \end{equation}
    The element $\zeta$ is the primitive form fixed by the opposite subspace $\L$.
\end{theorem}
The theorem above works in the ring of formal power series in $\bs$; however, it is known from the general theory that there is a unique analytic primitive form associated to the fixed opposite filtration. Then the formal primitive form $\zeta$ above is a series expansion of the analytic primitive form (cf. Section 4.3 of \cite{LLS}).

\subsection{BV algebra primitive form calculus}
For any $p \ge 0$, let $\CC[\bs]_p$ stand for the finite--dimensional subspace of $\CC[\bs]$ spanned by monomials of total degree $p$ in $s_1,\dots,s_\mu$. For any $\phi \in \CC[\bs]$, denote by $\phi_{(p)}$ the image of $\phi$ under the projection $\CC[\bs] \to \CC[\bs]_{p}$.

Fix a basis $\omega_1,\dots,\omega_\mu \in \H_F^{(0)}$.
For $\zeta \in \H_F^{(0)}$, let $\zeta_{(p)} \in \H_F^{(0)}$ be its $p$--homogeneous component in $\bs$. Namely,
\[
    \zeta = \sum_{\alpha=1}^\mu d_\alpha(\bs,\bz) \omega_\alpha
    \quad
    \Rightarrow
    \quad
    \zeta_{(p)} = \sum_{\alpha=1}^\mu \left( d_\alpha(\bs,\bz) \right)_{(p)}\omega_\alpha.
\]
Obviously we have $\zeta = \sum_{p \ge 0} \zeta_{(p)}$.

Denote by $\pi_{>0}$ and $\pi_{\le 0}$ the projections $\CC((z^{-1})) \to z\CC[z]$ and $\CC((z^{-1})) \to \CC[[z^{-1}]]$ given by
\begin{align*}
 \pi_{>0}(\sum_{k = -\infty}^m a_k z^k) = \sum_{k = 1}^m a_k z^k,
 \quad
 \pi_{\le 0}(\sum_{k = -\infty}^m a_k z^k) = \sum_{k = -\infty}^0 a_k z^k.
\end{align*}

\begin{theorem}\label{theorem: primitive form}
Let $\Phi_f^\omega$ be the $\p$--trivialization associated with the opposite subspace $\L$.
Then the primitive form $\zeta$ associated to $\L$ is found via the recursive formula for its components in the associated good basis:
\begin{align}
    \zeta_{(p)} = - \left[ \pi_{>0} \cdot (\Phi_f^\omega)^{-1} \Big( \sum_{a=1}^p \frac{(F-f)^a}{z^a a!} \zeta_{(p-a)}  \Big) \right]_f, \quad p \ge 1,
\end{align}
where $[-]_f$ stands for taking the classes in $\Jac(f)$.
\end{theorem}
\begin{proof}
    Using the isomorphism $\Upsilon_0$ of Eq.~\ref{eq: preBL as cohomology}, we can equivalently solve Eq.~\eqref{eq: main Jzeta} in $H^\ast(\PV((z^{-1})) \otimes \O_\S,\overline\p_f + z\p)$.
    Slightly abusing the notation, we use the same letters $J$ and $\zeta$ for the respective elements in $H^\ast(\PV((z^{-1})) \otimes \O_\S,\overline\p_f + z\p)$ and $H^\ast(\PV [z] \otimes \O_\S,\overline\p_F + z\p)$.

    Note that under $\Upsilon_0$ and $\Upsilon_1$, the exponential operator $e^{(F-f)/z}$ is mapped to the operator of multiplication by the $\PV^0[[z^{-1}]]$--element $e^{(F-f)/z}$.

    We have by construction $\Phi_f^\omega(J) = J$ because $\Phi_f^\omega$ acts identically on the good basis elements.
    Then Eq.~\eqref{eq: main Jzeta} is equivalent to
    \[
        J = \left[ \pi_{\le 0} \cdot (\Phi_f^\omega)^{-1} \Big( \sum_{k \ge 0} \frac{(F-f)^k}{z^k k!} \zeta \Big)  \right]_f
    \]
    where both $\frac{(F-f)^k}{z^k k!}$ and $\zeta$ inside the brackets are multiplied as $\PV^0$--elements.

    Then we have
    \[
        \left[ \pi_{> 0} \cdot (\Phi_f^\omega)^{-1} \Big( \sum_{k \ge 0} \frac{(F-f)^k}{z^k k!} \zeta \Big) ) \right]_f = 0.
    \]
    Using the linearity of $\Phi_f^\omega$ and the positivity ansatz on $\zeta$, we have
    \[
        (\Phi_f^\omega)^{-1} (\zeta) = \zeta = - \left[ \pi_{>0} \cdot (\Phi_f^\omega)^{-1} \Big( \sum_{k \ge 1} \frac{(F-f)^k}{z^k k!} \zeta \Big) \right]_f
    \]
    because $\Phi_f^\omega$ acts identically on the good basis elements. This concludes the proof because $\Phi_f^\omega$ does not depend on the variables $\bs$, and $F-f \in \CC[\bx] \otimes \CC[\bs]_1$.
\end{proof}
 
Related to our theorem is Theorem 5.16 of \cite{LLS}. Our theorem uses a completely different technique, whereas loc.cit. basically performs many computations in the Brieskorn lattice. We think that our theorem is also more efficient for computational purposes. A big advantage of our theorem is that one does not need to apply the exponential operator by computing the components $h^{(k)}_{\alpha\beta,m}$. In our theorem, this is just an element of the BV algebra. One more point that differs our result from that of Li, Li and Saito is that we work in the rings $\CC[\bs]_p$ rather than $\CC[\bx]_{\le p}$ as in loc.cit..

\section{R--matrix of a Dubrovin--Frobenius manifold}
It was first observed by B. Dubrovin (cf. Lemma~4.2 \cite{D2}) that under the action of a formal power series $R = \Id + \sum_{k \ge 1} z^k R_k$ with $R_k \in \mathrm{Mat}(\CC,\mu)$ (called gauge transform in loc.cit.), Dubrovin's connection~\ref{eq: Dubrovin connection} assumes a very simple form:
\[
    R^{-1} \cdot \widetilde \nabla_{\frac{d}{dz}} (R \cdot X) = \frac{d X}{dz} - z^{-2} B_0(X).
\]
This formal power series satisfies
\begin{equation}\label{eq: R-matrix main}
 [B_0, R_{m+1}] = (m + B_\infty) R_m, \quad m \ge 0.
\end{equation}
It was introduced by B. Dubrovin for classification purposes and later used by A. Givental in \cite{G04} in order to define the higher genus potential of the Dubrovin--Frobenius manifold.

\begin{proposition}[Proposition of Section~1.3 \cite{G04}]\label{prop: Givental-main}
For any Dubrovin--Frobenius manifold in a neighborhood of a semisimple point, the following holds:
 \begin{itemize}
  \item The solution $R$ to Eq.~\eqref{eq: R-matrix main} exists. It can be chosen to satisfy the ``symplectic'' condition
    \[
    R(z)\eta^{-1} R^T(-z) = \eta^{-1}.
    \]
  \item If $R$ is a homogeneous solution to Eq.~\eqref{eq: R-matrix main}:
  \[
    R_k = - \frac{1}{k} \iota_E d R_k, \quad \forall k \ge 1,
  \]
  then it satisfies the symplectic condition above and is unique.
 \end{itemize}
\end{proposition}

The unique solution $R$ of the proposition above will be called the {\it R--matrix of a Dubrovin--Frobenius manifold} (also called Givental's R--matrix). Its importance is justified by the following theorem.

\begin{theorem}[\cite{G01b, G04}]
    Let $(M,\circ,\eta,e)$ be a Dubrovin--Frobenius manifold.
    At a semisimple point $p \in M$, its potential is reconstructed from the topological theory potential of $(T_pM,\circ,\eta)$ via the Givental action of the R--matrix of the Dubrovin--Frobenius manifold.
\end{theorem}

It is important to note that the R--matrix of a Dubrovin--Frobenius manifold depends on the choice of the semisimple point $p$.

The coordinate $z$ in the R--matrix above is completely formal for a generic Dubrovin--Frobenius manifold. However, it coincides with the $z$ coordinate of $\H_F^{(0)}$ and $\nabla^\GM$ for Saito theory Dubrovin--Frobenius manifolds. Note that we have there $\wt(z) = 1$. Then the homogeneity condition of the R--matrix above is just $\wt R = 0$.

%

\subsection{R--matrix in PV}
We apply Proposition~\ref{prop: connections match} in order to rewrite the condition on the R--matrix of a Dubrovin--Frobenius manifold.

Denote $\Phi := \Upsilon_1 \Phi^\omega$. We have
\begin{align*}
    & \nabla_{\frac{\p}{\p z}}^\omega (R(a)) = \Phi^{-1} \cdot \nabla^\PV_{\frac{\p}{\p z}} \left( \Phi(R(a)) \right).
\end{align*}
Denote $\widetilde R := \Phi \cdot R$. The equality above rewrites as
\begin{align}\label{eq: widetilde R}
    & \widetilde R \left( \frac{\p a}{\p z} - z^{-2} Fa \right) = \frac{\p \widetilde R}{\p z}(a) + \widetilde R(\frac{\p a}{\p z}) - z^2 F \widetilde R (a).
    \\
    \Leftrightarrow \quad & [F, \widetilde R](a) = z^2 \frac{\p}{\p z} \widetilde R(a)
\end{align}
that should hold in $H^\ast(\PV\otimes\O_\S[z],\overline\p_F + z \p)$ for all $a$.

In what follows we construct the operator $\widetilde R$ solving the equation above explicitly. To do this we need some ingredients.


\begin{lemma}
In a neighborhood of a semisimple point $s \in \S$, there is an operator $\A = \sum_{k \ge 0} A_k z^k \in \End(H^\ast( \PV \otimes \O_\S[z], \bar \p_F + z \p))$ such that $F \cdot \A = \Id$.

\end{lemma}
\begin{proof}
    The action of the operator $F \cdot$ in $H^\ast( \PV \otimes \O_\S, \bar \p_F )$ is given by multiplication by the Euler vector field. This operator is known to be invertible at a semisimple point. Denote this inverse by $A_0$. The higher $A_k$ are found recursively by the following procedure. It's enough to solve
    \[
        (\Phi^{\omega})^{-1} \left( \sum_{k \ge 0} A_k z^k F \right) - 1 \equiv 0 \quad \text{in} \ H^\ast( \PV \otimes \O_\S, \bar \p_F ) \otimes \CC[[z]].
    \]
    Let $(\Phi^{\omega})^{-1} = \sum_{p \ge 0} \overline \Phi_p z^p$. Then the $k$--th order in $z$ of the equation above reads
    \[
        0 \equiv \sum_{a+b = k} \overline \Phi_a (A_b F) = A_k F + \sum_{a=1}^k \overline \Phi_a (A_{k-a} F).
    \]
    It allows to express $A_k$ via $A_0,\dots,A_{k-1}$ becasue $F \cdot$ is invertible.

\end{proof}

One notes immediately that every $A_k$ is homogeneous of degree $-(k+1)$ for all $k \ge 0$.

\begin{example}
 For $F = \frac{1}{3} x_1^3 + s_2 x_1 + s_1$, we have
 \begin{align*}
    \A &= \frac {9 s_1} {4 s_2^3 + 9 s_1^2} -\frac {6 s_2} {4 s_2^3 + 9 s_1^2} x_ 1 + z\left (\frac {27 s_1^2 - 24 s_2^3 } {\left (4 s_2^3 +
        9 s_1^2 \right)^2} - \frac {54 s_1 s_2 } {\left (4 \
s_2^3 + 9 s_1^2 \right)^2} x_ 1 \right)
    \\
    &+ z^2\left (
    \frac {9\left (9 s_1^3 - 32 s_1 s_2^3 \right)} {\left (4 s_2^3 + 9 s_1^2 \right)^3}
    + \frac {6 s_2\left (8 s_2^3 - 63 s_1^2 \right)} {\left (4 s_2^3 + 9 s_1^2 \right)^3}x_ 1
    \right) + O(z^3),
 \end{align*}
 the semisimplicity condition is equivalent to $4 s_2^3 + 9 s_1^2 \neq 0$.
\end{example}

For all $k \ge 1$, let $B_k \in \End(H^\ast( \PV \otimes \O_\S[z], \bar \p_F + z \p))$.
In a neighborhood of a semisimple point $s \in \S$, consider  $\widetilde R^o$ acting by
\[
    \widetilde R^o(a) := a + z A_0\cdot a - \sum_{k \ge 2} z^k \ \wt(a) B_k(a).
\]
Then $\widetilde R^o$ is homogeneous of degree $0$. It satisfies Eq.~\eqref{eq: widetilde R} if and only if $B_k$ satisfy in $H^\ast( \PV \otimes \O_\S[z], \bar \p_F + z \p)$
\[
    \sum_{k \ge 2} F \cdot B_k(a) z^k \equiv - \sum_{k \ge 1} z^{k+1} \ k B_k(a) \wt(a)
\]
where we denote $B_1 := - A_0 \cdot$. This is equivalent to
\[
 \sum_{k \ge 2} B_k(a) z^k \equiv  - \A \cdot \sum_{k \ge 1} z^{k+1} \ k B_k(a) \wt(a).
\]
This equality allows us to find $B_k$ recursively. We have $B_2(a) = - A_0 \cdot B_1(a) \wt(a)$ mod $z$ and
\begin{align*}
    & B_p(a) z^p \equiv  - \A \cdot \sum_{k = 1}^{p-1} z^{k+1} \ k B_k(a) \wt(a) - \sum_{k = 2}^{p-1} B_k(a) z^k + O(z^{p+1})
\end{align*}
for all $p \ge 2$. Note that in the previous equality we have $B_p$ only on the left-hand side and only known operators on the right-hand side.

For a fixed good basis, consider the polynomials $\phi_\bullet$ as in Eq.~\eqref{eq: omega_bullet to phi_bullet} satisfying $\omega_a \equiv [ \phi_a d^N \bx]$ mod $z \H_F^{(0)}$.

The main theorem of the current paper is the following.
\begin{theorem}\label{theorem: R--matrix}
    The operator $(\Phi^\omega )^{-1} \widetilde R^o$ written in the basis $\phi_1,\dots,\phi_\mu$ is the R--matrix of the Dubrovin--Frobenius manifold $M^\omega_F$.
\end{theorem}
\begin{proof}
    By the lemma above, $\widetilde R^o$ is homogeneous of degree $0$. The resolution operator $\Phi^{top}$ is homogeneous of degree $0$ by its construction. The resolution operator $\Phi^{\omega}$ is homogeneous of degree $0$ because a good basis is required to be homogeneous. We conclude that $R$ is homogeneous of degree $0$ in $H^\ast(\PV\otimes\O_\S,\overline\p_F)[z]$ as well.

    Compute the components of $\widetilde R^o$ in the basis $\phi_\alpha$.
    By its construction, the operator $R^o$ satisfies Eq.~\eqref{eq: R-matrix main} in $H^\ast(\PV\otimes\O_\S,\overline\p_F)[z]$.
    This now follows from Proposition~\ref{prop: Givental-main} that it is the R--matrix of the Dubrovin--Frobenius manifold.
\end{proof}

\begin{corollary}
    The BV--algebra $(\PV \otimes \O_\S[z],\overline\p_F,\p)$ correlators of \cite{KMS} computed in the BV--operator trivialization $\widetilde R^o \Phi^\omega$ coincide with the Saito theory correlators.
\end{corollary}
\begin{proof}
 This follows immediately from Theorem 6.3 in \cite{KMS}.
\end{proof}

\bibliographystyle{unsrt}

\end{document}